\documentclass[12pt]{amsart}
\usepackage[active]{srcltx}
\usepackage{calc,amssymb,amsthm,amsmath,amscd, eucal,ulem}
\usepackage{alltt}
\usepackage[left=1.35in,top=1.25in,right=1.35in,bottom=1.25in]{geometry}
\RequirePackage[dvipsnames,usenames]{color}

\input{kmacros3.sty}
\input{xy}
\xyoption{all}
\usepackage{tikz}

\numberwithin{equation}{theorem}

\newcommand{\D}{\displaystyle}

\renewcommand{\m}{\mathfrak{m}}
\renewcommand{\n}{\mathfrak{n}}

\DeclareMathOperator{\sspan}{span}

\usepackage{fullpage}

\usepackage{setspace}
\usepackage{hyperref}

\usepackage{enumerate}

\usepackage{graphicx}

\usepackage[all,cmtip]{xy}
\usepackage{verbatim}

\theoremstyle{theorem}

\begin{document}
\title{Finiteness properties of local cohomology for $F$-pure local rings}
\author{Linquan Ma}
\address{Department of Mathematics\\ University of Michigan\\ Ann Arbor\\ Michigan 48109}
\email{lquanma@umich.edu}
\maketitle
\begin{abstract}
In this paper, we show that for an $F$-pure local ring $(R,\m)$, all
local cohomology modules $H_{\m}^i(R)$ have finitely many Frobenius
compatible submodules. This answers positively the open question raised by F. Enescu and M. Hochster in
\cite{MelandFlorianFrobenius} (see also \cite{Florianafinitenessconditon}, where it was stated as a
conjecture when $R$ is Cohen-Macaulay). We also prove that when $(R,\m)$ is excellent and $F$-pure on the punctured spectrum, all local cohomology modules $H_{\m}^i(R)$ have finite length in the category of $R$-modules with Frobenius action. Finally, we show that the property that all $H_\m^i(R)$ have finitely many Frobenius compatible submodules passes to localizations.
\end{abstract}

\section{Introduction}
Let $(R,\m)$ be a local ring of equal characteristic $p>0$ of
dimension $d$. There is a natural action of the Frobenius
endomorphism of $R$ on each of its local cohomology modules
$H_{\m}^i(R)$. We call an $R$-submodule $N$ of $H_{\m}^i(R)$
{\it $F$-compatible} if $F$ maps $N$ into itself. Our main interest is to
understand when a local ring $(R,\m)$ has the property that there
are only finitely many $F$-compatible submodules for each
$H_{\m}^i(R)$, $0\leq i\leq d$. Rings with this property are called {\it FH-finite} and have been studied in detail in \cite{MelandFlorianFrobenius}. It was proved in \cite{MelandFlorianFrobenius} that an $F$-pure Gorenstein ring is FH-finite. This also follows from results in \cite{SharpGradedAnnihilatorsOfModulesOverTheFrobeniusSkewPolynomialRing}. Recently, it was
proved in \cite{Florianafinitenessconditon} that if an $F$-injective
Cohen-Macaulay ring $R$ admits a canonical ideal $I\cong \omega_R$ such that $R/I$ is $F$-pure, then $R$ is
FH-finite. We notice that all these hypothesis imply $R$ is $F$-pure.
In fact, it was asked in \cite{MelandFlorianFrobenius} (and this was later
conjectured in \cite{Florianafinitenessconditon} for Cohen-Macaulay rings) whether the
$F$-pure property itself is enough for FH-finiteness. We provide a
positive answer to this question. We actually prove a stronger
result, which says that $F$-pure will imply {\it stably FH-finite}, this means that $R$ and all power series rings over $R$ are FH-finite.
Our first main result is the following, proved in Section 3:
\begin{theorem}
\label{main theorem 1} Let $(R,\m)$ be an $F$-pure local ring. Then
$R$ and all power series rings over $R$ are FH-finite, that is, $R$ is stably FH-finite.
\end{theorem}

We shall also study the problem of determining conditions under which the local cohomology modules $H_\m^i(R)$ have finite length in the category of $R$-modules with Frobenius action. In fact, rings with this property are said to have {\it FH-finite length} and have been studied in \cite{MelandFlorianFrobenius}. Inspired by the theory of {\it Cartier modules} introduced by M. Blickle and G. B\"ocklein in \cite{BlickleBoeckleCartierModulesFiniteness} and the $\Gamma$-construction introduced by M. Hochster and C. Huneke in  \cite{HochsterHunekeFRegularityTestElementsBaseChange}, we obtain our second and third main result, proved in Section 5:
\begin{theorem}
\label{main theorem 2} Let $(R,\m)$ be an excellent local ring such that $R_P$ is $F$-pure for every $P\in \Spec R-\{\m\}$. Then $R$ has FH-finite length.
\end{theorem}
\begin{theorem}
\label{main theorem 3}
Let $(R,\m)$ be a local ring that has FH-finite length (resp. is FH-finite or stably FH-finite). Then the same holds for $R_P$ for every $P\in \Spec{R}$.
\end{theorem}

Throughout this paper we always assume $(R,\m)$ is a Noetherian local
ring of equal characteristic $p>0$, except at the beginning of
Section 3 and the end of Section 4, where the hypothesis will be stated clearly. In Section 2 we
start with some definitions and properties of $F$-pure and $F$-split
rings, and we recall some theorems on the FH-finiteness and
anti-nilpotency of $H_\m^i(R)$ proved in \cite{MelandFlorianFrobenius}. In Section 3 we prove Theorem
\ref{main theorem 1}. In Section 4 we prove a key theorem, Theorem \ref{FH-finite length is equiv to stably FH-finite on punctured spec} which implies both Theorem \ref{main theorem 2} and Theorem \ref{main theorem 3} when $R$ is complete and $F$-finite and we also prove the dual form of Theorem \ref{main theorem 1} and Theorem \ref{main theorem 2} in terms of Cartier modules. In Section 5, we make use of the $\Gamma$-construction in \cite{HochsterHunekeFRegularityTestElementsBaseChange} to prove Theorem \ref{main theorem 2} and Theorem \ref{main theorem 3} in full generality. Finally, in Section 6, we study some examples. Some of our techniques are inspired by the work in \cite{SinghandUlilocalcohomologyandpure}.

\section{FH-finiteness and anti-nilpotency}
Recall that a map of $R$-modules $N\rightarrow N'$ is {\it pure} if
for every $R$-module $M$ the map $N\otimes_RM\rightarrow
N'\otimes_RM$ is injective. This implies that $N\rightarrow N'$ is
injective, and is weaker than the condition that $0\rightarrow
N\rightarrow N'$ be split. A local ring $(R,\m)$ is called {\it $F$-pure}
(respectively, {\it $F$-split}) if the Frobenius endomorphism $F$:
$R\rightarrow R$ is pure (respectively, split). Evidently, an
$F$-split ring is $F$-pure and an $F$-pure ring is reduced. When $R$
is either $F$-finite or complete, $F$-pure and $F$-split are
equivalent ({\it cf}. Discussion 2.6 in
\cite{MelandFlorianFrobenius}).

We will use some notations introduced in
\cite{MelandFlorianFrobenius}. We say an $R$-module $M$ is an
{\it $R\{F\}$-module} if there is a Frobenius action $F$: $M\rightarrow
M$ such that for all $u\in M$, $F(ru)=r^pu$. This is same as saying
that $M$ is a left module over the ring $R\{F\}$, which may be viewed as
a noncommutative ring generated over $R$ by the symbols
$1,F,F^2,\dots$ by requiring that $Fr=r^pF$ for $r\in R$. We say
$N$ is an {\it $F$-compatible} submodule of $M$ if $F(N)\subseteq N$. Note
that this is equivalent to requiring that $N$ be an
$R\{F\}$-submodule of $M$. We say an $R\{F\}$-module $M$ is {\it $F$-nilpotent} if $F^e(M)=0$ for some $e$.

The Frobenius endomorphism on $R$ induces a natural Frobenius action
on each local cohomology module $H_\m^i(R)$ (see Discussion 2.2 and
2.4 in \cite{MelandFlorianFrobenius} for a detailed explanation of
this). We say a local ring is {\it $F$-injective} if $F$ acts injectively
on all of the local cohomology modules of $R$ with support in $\m$.
This holds if $R$ is $F$-pure.
\begin{definition}[{\it cf.} Definition 2.5 in \cite{MelandFlorianFrobenius}]
A local ring $(R,\m)$ of dimension $d$ is called {\it FH-finite} if for
all $0\leq i\leq d$, there are only finitely many $F$-compatible
submodules of $H_\m^i(R)$. We say $(R,\m)$ has {\it FH-finite length} if for each $0\leq i\leq d$, $H_\m^i(R)$ has finite length in the category of $R\{F\}$-modules.
\end{definition}

In \cite{MelandFlorianFrobenius}, F. Enescu and M. Hochster also
introduced the anti-nilpotency condition for $R\{F\}$-modules, which
turns out to be very useful and has close connections with the
finiteness properties of local cohomology modules. In fact, it is
proved in \cite{MelandFlorianFrobenius} that the anti-nilpotency of
$H_\m^i(R)$ for all $i$ is equivalent to the condition that all power series rings
over $R$ be FH-finite.
\begin{definition}
Let $(R,\m)$ be a local ring and let $W$ be an $R\{F\}$-module. We
say $W$ is {\it anti-nilpotent} if for every $F$-compatible submodule
$V\subseteq W$, $F$ acts injectively on $W/V$.
\end{definition}
\begin{theorem}[{\it cf.} Theorem 4.15 in \cite{MelandFlorianFrobenius}]
\label{anti-nilpotency is equiv to stably FH finite} Let $(R,\m)$ be
a local ring and let $x_1,\dots, x_n$ be formal power series
indeterminates over $R$. Let $R_0=R$ and $R_n=R[[x_1,\dots, x_n]]$.
Then the following conditions on $R$ are equivalent:
\begin{enumerate}
\item All local cohomology modules $H_\m^i(R)$ are anti-nilpotent.
\item The ring $R_n$ is FH-finite for every $n$.
\item $R_1\cong R[[x]]$ has FH-finite length.
\end{enumerate}
When $R$ satisfies these equivalent conditions, we call it {\it
stably} FH-finite.
\end{theorem}

We will also need some results in \cite{BlickleBoeckleCartierModulesFiniteness} about Cartier modules. We recall some definitions in \cite{BlickleBoeckleCartierModulesFiniteness}. The definitions and results in \cite{BlickleBoeckleCartierModulesFiniteness} work for schemes and sheaves, but we will only give the corresponding definitions for local rings for simplicity (we will not use the results on schemes and sheaves).
\begin{definition}
A {\it Cartier module} over $R$ is an $R$-module equipped with a $p^{-1}$ linear map $C_M$: $M\rightarrow M$, that is an additive map satisfying $C(r^px)=rC(x)$ for every $r\in R$ and $x\in M$. A Cartier module $(M,C)$ is called {\it nilpotent} if $C^{e}(M)=0$ for some $e$.
\end{definition}
\begin{remark}
\begin{enumerate}
\item A Cartier module is nothing but a right module over the ring $R\{F\}$ (see \cite{SharpandYoshinoRightandleftmodulesoverFrobeniuspolynomialring} for corresponding properties of right $R\{F\}$-modules).
\item If $(M,C)$ is a Cartier module, then $C_P$: $M_P\rightarrow M_P$ defined by \[C_P(\frac{x}{r})=\frac{C(r^{p-1}x)}{r}\] for every $x\in M$ and $r\in R-P$ gives $M_P$ a Cartier module structure over $R_P$.
\end{enumerate}
\end{remark}
We end this section with a simple lemma which will reduce most problems to the $F$-split case (recall that for complete local rings, $F$-pure is equivalent to $F$-split).
\begin{lemma}
\label{reduce to F-split} Let $(R,\m)$ be a local ring. Then $R$ has FH-finite length (resp. is FH-finite or stably FH-finite) if and only if $\widehat{R}$ has FH-finite length (resp. is FH-finite or stably FH-finite).
\end{lemma}
\begin{proof}
Since completion does not affect either what the local cohomology modules are nor what the action of Frobenius is, the lemma follows immediately.
\end{proof}

\section{$F$-pure implies stably FH-finite}
In order to prove the main result, we begin with some simple Lemmas
\ref{basic lemma}, \ref{lemma 1}, \ref{lemma 2} and a Proposition
\ref{key prop} which are characteristic free. In fact, in all these
lemmas we only need to assume $I$ is a finitely generated ideal in a (possibly non-Noetherian) ring $R$ so
that the \v{C}ech complex characterization of local cohomology can
be applied (the proof will be the same). However, we only state these results when $R$ is Noetherian.
\begin{lemma}
\label{basic lemma} Let $R$ be a Noetherian ring, $I$ be an ideal of
$R$ and $M$ be any $R$-module. We have a natural map:
\begin{equation*}
M\otimes_RH_I^i(R)\xrightarrow{\phi} H_I^i(M)
\end{equation*}
Moreover, when $M=S$ is an $R$-algebra, $\phi$ is $S$-linear.
\end{lemma}
\begin{proof}
Given maps of $R$-modules $L_1\xrightarrow{\alpha}
L_2\xrightarrow{\beta} L_3$ and $M\otimes_R
L_1\xrightarrow{id\otimes\alpha}
M\otimes_RL_2\xrightarrow{id\otimes\beta} M\otimes_RL_3$ such that
$\beta\circ\alpha=0$, there is a natural map:
\begin{equation*}
M\otimes_R \D\frac{\ker\beta}{\im\alpha}\rightarrow
\frac{\ker(id\otimes\beta)}{\im(id\otimes\alpha)}
\end{equation*}
sending $m\otimes \overline{z}$ to $\overline{m\otimes z}$. Now the
result follows immediately by the \v{C}ech complex characterization
of local cohomology.
\end{proof}
\begin{lemma}
\label{lemma 1} Let $R$ be a Noetherian ring, $S$ be an $R$-algebra,
and $I$ be an ideal of $R$. We have a commutative diagram:
\[  \xymatrix{
    &     S\otimes_RH_{I}^i(R)  \ar[d]^{\phi}   \\
H_{I}^i(R)    \ar[ur]^{j_2}  \ar[r]^-{j_1} & H_{I S}^i(S) 
} \] where $j_1, j_2$ are the natural maps induced by $R\rightarrow
S$, in particular $j_2$ sends $z$ to $1\otimes z$.
\end{lemma}
\begin{proof}
This is straightforward to check.
\end{proof}
\begin{lemma}
\label{lemma 2} Let $R$ be a Noetherian ring and $S$ be an
$R$-algebra such that the inclusion $\iota$: $R\hookrightarrow S$
splits. Let $\gamma$ be the splitting $S\rightarrow
R$. Then we have a commutative diagram:
\[  \xymatrix{
    &     S\otimes_RH_{I}^i(R)  \ar[dl]_{q_2}   \ar[d]^{\phi}   \\
H_{I}^i(R)     & H_{I S}^i(S) \ar[l]_{q_1} 
} \] where $q_1, q_2$ are induced by $\gamma$, in particular $q_2$
sends $s\otimes z$ to $\gamma(s)z$.
\end{lemma}
\begin{proof}
We may identify $S$ with $R\oplus W$ and $R\hookrightarrow S$ with
$R\hookrightarrow R\oplus W$ which sends $r$ to $(r,0)$, and
$S\rightarrow R$ with $R\oplus W\rightarrow R$ which sends $(r,w)$
to $r$ (we may take $W$ to be the $R$-submodule of $S$ generated by
$s-\iota\circ\gamma(s)$). Under this identification, we have:
\[
S\otimes_RH_I^i(R)=H_I^i(R)\oplus W\otimes_RH_I^i(R)
\]
\[
H_{IS}^i(S)=H_I^i(R)\oplus H_I^i(W)
\]
and $q_1, q_2$ are just the projections onto the first factors. Now
the conclusion is clear because by Lemma \ref{basic lemma}, $\phi$:
$S\otimes_RH_I^i(R)\rightarrow H_{IS}^i(S)$ is the identity on
$H_I^i(R)$ and sends $W\otimes_RH_I^i(R)$ to $H_I^i(W)$.
\end{proof}

\begin{proposition}
\label{key prop} Let $R$ be a Noetherian ring and $S$ be an
$R$-algebra such that $R\hookrightarrow S$ splits.
Let $y$ be an element in $H_I^i(R)$ and $N$ be a submodule of
$H_I^i(R)$. If the image of $y$ is in the $S$-span of the image of
$N$ in $H_{I S}^i(S)$, then $y\in N$.
\end{proposition}
\begin{proof}
We know there are two commutative diagrams as in Lemma \ref{lemma 1}
and \ref{lemma 2} (note that here $j_1$ and $j_2$ are inclusions
since $R\hookrightarrow S$ splits). We use $\gamma$ to denote the
splitting $S\rightarrow R$. The condition says that $j_1(y)=\sum
s_k\cdot j_1(n_k)$ for some $s_k\in S$ and $n_k\in N$. Applying
$q_1$ we get:
\begin{eqnarray*}
y&=&q_1\circ j_1(y)\\
&=&\sum q_1(s_k\cdot j_1(n_k))\\
&=&\sum q_1(s_k\cdot\phi\circ j_2(n_k))\\
&=&\sum q_1 \circ \phi (s_k \cdot j_2( n_k))\\
&=&\sum q_2(s_k \otimes n_k)\\
&=&\sum \gamma(s_k)\cdot n_k\in N
\end{eqnarray*}
where the first equality is by definition of $q_1$, the third
equality is by Lemma \ref{lemma 1}, the fourth equality is because
$\phi$ is $S$-linear, the fifth equality is by Lemma \ref{lemma 2}
and the definition of $j_2$ and the last equality is by the
definition of $q_2$. This finishes the proof.
\end{proof}

Now we return to the situation in which we are interested. We assume
$(R,\m)$ is a Noetherian local ring of equal characteristic $p>0$.
We first prove an immediate corollary of Proposition \ref{key prop},
which explains how FH-finite and stably FH-finite properties behave
under split maps.
\begin{corollary}
\label{FH-finite for split couple}
Suppose $(R,\m)\hookrightarrow(S,\n)$ is split and ${\m S}$ is
primary to $\n$. Then if $S$ is FH-finite (respectively, stably
FH-finite), so is $R$.
\end{corollary}
\begin{proof}
First notice that, when $R\hookrightarrow S$ is split, so is
$R[[x_1,\dots,x_n]]\hookrightarrow S[[x_1,\dots,x_n]]$. So it
suffices to prove the statement for FH-finite. Since $\m S$ is
primary to $\n$, for every $i$, we have a natural commutative
diagram:
\[  \xymatrix{
   H_\m^i(R) \ar[d]^{F}  \ar[r] &    H_{\n}^i(S)  \ar[d]^{F}   \\
H_\m^i(R)    \ar[r]   & H_{\n}^i(S) 
} \] where the horizontal maps are induced by the inclusion
$R\hookrightarrow S$, and the vertical maps are the Frobenius
action. It is straightforward to check that if $N$ is an
$F$-compatible submodule of $H_\m^i(R)$, then the $S$-span of $N$ is
also an $F$-compatible submodule of $H_\n^i(S)$.

If $N_1$ and $N_2$ are two different $F$-compatible submodules of
$H_\m^i(R)$, then their $S$-span in $H_\n^i(S)$ must be different by
Proposition \ref{key prop}. But since $S$ is FH-finite, each
$H_\n^i(S)$ only has finitely many $F$-compatible submodules. Hence
so is $H_\m^i(R)$. This finishes the proof.
\end{proof}

Now we start proving our main result. First we prove a lemma:
\begin{lemma}
\label{key lemma} Let $W$ be an $R\{F\}$-module. Then $W$ is
anti-nilpotent if and only if for every $y\in W$, $y\in
\sspan_R\langle F(y), F^2(y), F^3(y),\dots\dots\rangle$ .
\end{lemma}
\begin{proof}
Suppose $W$ is anti-nilpotent. For each $y\in W$,
$V:=\sspan_R\langle F(y), F^2(y), F^3(y),\dots\dots\rangle$ is an
$F$-compatible submodule of $W$. Hence, $F$ acts injectively on $W/V$
by anti-nilpotency of $W$. But clearly $F(\overline{y})=0$ in $W/V$,
so $\overline{y}=0$, so $y\in V$.

For the other direction, suppose there exists some $F$-compatible
submodule $V\subseteq W$ such that $F$ does not act injectively on
$W/V$. We can pick some $y \notin V$ such that $F(y)\in V$. Since
$V$ is an $F$-compatible submodule and $F(y)\in V$, $\sspan_R\langle
F(y), F^2(y), F^3(y),\dots\dots\rangle\subseteq V$. So $y\in
\sspan_R\langle F(y), F^2(y), F^3(y),\dots\dots\rangle \subseteq V$
which is a contradiction.
\end{proof}
\begin{theorem}
\label{F-split implies anti-nilpotency} Let $(R,\m)$ be a local ring
which is $F$-split. Then $H_{\m}^i(R)$ is anti-nilpotent for every
$i$.
\end{theorem}
\begin{proof}
By Lemma \ref{key lemma}, it suffices to show for every $y\in
H_{\m}^i(R)$, we have \[y\in \sspan_R\langle F(y), F^2(y),
F^3(y),\dots\dots\rangle.\] Let $N_j=\sspan_R\langle
F^j(y),F^{j+1}(y)\dots\dots\rangle$, consider the descending chain:
\begin{equation*}
N_0 \supseteq N_1\supseteq N_2\supseteq\cdots\cdots \supseteq
N_j\supseteq\cdots\cdots
\end{equation*}
Since $H_m^{i}(R)$ is Artinian, this chain stabilizes, so there
exists a smallest $e$ such that $N_e=N_{e+1}$. If $e=0$ we are done.
Otherwise we have $F^{e-1}(y)\notin N_e$. Since $R$ is $F$-split, we
apply Proposition \ref{key prop} to the Frobenius map
$R\xrightarrow{r\rightarrow r^p}R=S$ (and $I=\m$). In order to make
things clear we use $S$ to denote the target $R$, but we keep in
mind that $S=R$.

From Proposition \ref{key prop} we know that the image of
$F^{e-1}(y)$ is not contained in the $S$-span of the image of $N_e$
under the map $H_\m^i(R)\rightarrow H_{\m S}^i(S) \cong H_\m^i(R)$.
But this map is exactly the Frobenius map on $H_\m^i(R)$, so the
image of $F^{e-1}(y)$ is $F^{e}(y)$, and after identifying $S$ with
$R$, the $S$-span of the image of $N_e$ is the $R$-span of
$F^{e+1}(y), F^{e+2}(y), F^{e+3}(y),\dots\dots$ which is $N_{e+1}$.
So $F^{e}(y)\notin N_{e+1}$, which contradicts our choice of $e$.
\end{proof}
\begin{theorem}
\label{F-pure implies FH-finite} Let $(R,\m)$ be an $F$-pure local
ring. Then $R$ and all power series rings over $R$ are FH-finite
(i.e. $R$ is stably FH-finite).
\end{theorem}
\begin{proof}
By Lemma \ref{reduce to F-split}, we may assume $R$ is $F$-split.
Now the result is clear from Theorem \ref{F-split implies
anti-nilpotency} and Theorem \ref{anti-nilpotency is equiv to stably
FH finite}.
\end{proof}

\section{FH-finite length and stable FH-finiteness in the complete $F$-finite case}

In this section, we show for a complete and $F$-finite local ring $(R,\m)$, the condition that $R_P$ be stably FH-finite for all $P\in \Spec{R}-\{\m\}$ is equivalent to the condition that $R$ have FH-finite length. First we recall the following important theorem of Lyubeznik:
\begin{theorem}[{\it cf.} Theorem 4.7 in \cite{LyubeznikFModulesApplicationsToLocalCohomology} or Theorem 4.7 in \cite{MelandFlorianFrobenius}]
\label{Lyubeznik's result}
Let $W$ be an $R\{F\}$-module which is Artinian as an $R$-module. Then $W$ has a finite filtration
\begin{equation}
\label{Lyubeznik's filtration}
0= L_0\subseteq N_0\subseteq L_1\subseteq N_1\subseteq \cdots\subseteq L_s\subseteq N_s=W
\end{equation}
by Frobenius compatible submodules of $W$ such that every $N_j/L_j$ is $F$-nilpotent, while every $L_j/N_{j-1}$ is simple in the category of $R\{F\}$-modules, with a nonzero Frobenius action. The integer $s$ and the isomorphism classes of the modules $L_j/N_{j-1}$ are invariants of $W$.
\end{theorem}

The following proposition in \cite{MelandFlorianFrobenius} characterizes being anti-nilpotent and having finite length in the category of $R\{F\}$-modules in terms of Lyubeznik's filtration:

\begin{proposition}[{\it cf.} Proposition 4.8 in \cite{MelandFlorianFrobenius}]
\label{characterization of anti-nilpotency interms of Lyubeznik's filtration}
Let the notations and hypothesis be as in Theorem \ref{Lyubeznik's result}. Then:
\begin{enumerate}
\item $W$ has finite length as an $R\{F\}$-module if and only if each of the factors $N_j/L_j$ has finite length in the category of $R$-modules.
\item $W$ is anti-nilpotent if and only if in some (equivalently, every) filtration, the nilpotent factors $N_j/L_j=0$ for every $j$.
\end{enumerate}
\end{proposition}

\begin{remark}
It is worth pointing out that an Artinian $R$-module $W$ is Noetherian over $R\{F\}$ if and only if in some (equivalently, every) filtration as in Theorem \ref{Lyubeznik's result}, each of the factors $N_j/L_j$ is Noetherian as an $R$-module (same proof as Proposition 4.8 in \cite{MelandFlorianFrobenius}). So $W$ is Noetherian over $R\{F\}$ if and only if it has finite length as an $R\{F\}$-module. Hence $R$ has FH-finite length if and only if all local cohomology modules $H_\m^i(R)$ are Noetherian $R\{F\}$-modules.
\end{remark}

We also need the following important theorem in \cite{BlickleBoeckleCartierModulesFiniteness} which relates $R\{F\}$-modules and Cartier modules. This result was also proved independently by R. Y. Sharp and Y. Yoshino in \cite{SharpandYoshinoRightandleftmodulesoverFrobeniuspolynomialring} in the language of left and right $R\{F\}$-modules.

\begin{theorem}[{\it cf.} Proposition 5.2 in \cite{BlickleBoeckleCartierModulesFiniteness} or Corollary 1.21 in \cite{SharpandYoshinoRightandleftmodulesoverFrobeniuspolynomialring}]
\label{matlis dual induce equivalence}
Let $(R,\m)$ be complete, local and $F$-finite. Then Matlis duality induces an equivalence of categories between $R\{F\}$-modules which are Artinian as $R$-modules and Cartier modules which are Noetherian as $R$-modules. The equivalence preserves nilpotence.
\end{theorem}

Now we start proving the main theorem of this section. We will use ${}^{\vee}$ to denote the Matlis dual with respect to $\m$ and ${}^{\vee_P}$ to denote the Matlis dual with respect to $PR_P$. It follows directly from local duality that for $(R,\m)$ a complete local ring, $(H_\m^i(R)^{\vee})_P^{\vee_P}\cong H_{PR_P}^{i-\dim R/P}(R_P)$. We begin by proving a lemma on Cartier modules.

\begin{lemma}
\label{simple cartier module localize}
We have the following:
\begin{enumerate}
\item If $M$ is a nilpotent Cartier module over $R$, then $M_P$ is a nilpotent Cartier module over $R_P$
\item If $(M, C)$ is a simple Cartier module over $R$ with a nontrivial $C$-action, then $(M_P, C_P)$ is a simple Cartier module over $R_P$, and if $M_P\neq 0$, then the $C_P$-action is also nontrivial.
\end{enumerate}
\end{lemma}
\begin{proof}
$(1)$ is obvious, because if $C^e$ kills $M$, then $C^e_P$ kills $M_P$. Now we prove $(2)$. Let $N$ be a Cartier $R_P$ submodule of $M_P$. Consider the contraction of $N$ in $M$, call it $N'$. Then it is easy to check that $N'$ is a Cartier $R$-submodule of $M$. So it is either $0$ or $M$ because $M$ is simple. But if $N'=0$ then $N=0$ and if $N'=M$ then $N=M_P$ because $N$ is an $R_P$-submodule of $M_P$. This proves $M_P$ is simple as a Cartier module over $R_P$. To see the last assertion, notice that if $M$ is a simple Cartier module with a nontrivial $C$-action, then $C$: $M\rightarrow M$ must be surjective: otherwise the image would be a proper Cartier submodule. Hence $C_P$: $M_P\rightarrow M_P$ is also surjective. But we assume $M_P\neq 0$, so $C_P$ is a nontrivial action.
\end{proof}

\begin{theorem}
\label{FH-finite length is equiv to stably FH-finite on punctured spec}
Let $(R,\m)$ be a complete and $F$-finite local ring. Then the following conditions are equivalent:
\begin{enumerate}
\item $R_P$ is stably FH-finite for every $P\in \Spec{R}-\{\m\}$.
\item $R$ has FH-finite length.
\end{enumerate}
\end{theorem}
\begin{proof}
By Theorem \ref{Lyubeznik's result}, for every $H_\m^i(R)$, $0\leq i\leq d$, we have a filtration
\begin{equation}
\label{filtration of R[F]-modules}
0= L_0\subseteq N_0\subseteq L_1\subseteq N_1\subseteq \cdots\subseteq L_s\subseteq N_s=H_\m^i(R)
\end{equation}
of $R\{F\}$-modules such that every $N_j/L_j$ is $F$-nilpotent while every $L_j/N_{j-1}$ is simple in the category of $R\{F\}$-modules, with nontrivial $F$-action. Now we take the Matlis dual of the above filtration (\ref{filtration of R[F]-modules}), we have \[H_\m^i(R)^{\vee}=N_s^{\vee}\twoheadrightarrow L_s^{\vee}\twoheadrightarrow \cdots \twoheadrightarrow N_0^{\vee}\twoheadrightarrow L_0^{\vee}= 0\] such that each $\ker(L_j^{\vee}\twoheadrightarrow N_{j-1}^{\vee})$ is Noetherian as an $R$-module and is simple as a Cartier module with nontrivial $C$-action, and each $\ker(N_j^{\vee}\twoheadrightarrow L_j^{\vee})$ is Noetherian as an $R$-module and is nilpotent as a Cartier module by Lemma \ref{matlis dual induce equivalence}. When we localize at $P\neq \m$, we get \[(H_\m^i(R)^{\vee})_P=(N_s^{\vee})_P\twoheadrightarrow (L_s^{\vee})_P\twoheadrightarrow \cdots \twoheadrightarrow (N_0^{\vee})_P\twoheadrightarrow(L_0^{\vee})_P= 0\] with each $\ker((L_j^{\vee})_P\twoheadrightarrow (N_{j-1}^{\vee})_P)$ a simple Cartier module over $R_P$ whose $C_P$ action is nontrivial if it is nonzero, and each $\ker((N_j^{\vee})_P\twoheadrightarrow (L_{j}^{\vee})_P)$ a nilpotent Cartier module over $R_P$ by Lemma \ref{simple cartier module localize}. Now, when we take the Matlis dual over $R_P$, we get a filtration of $R_P\{F\}$-modules
\begin{equation}
\label{dualized filtration}
0= L'_0\subseteq N'_0\subseteq  L'_1\subseteq \cdots\subseteq L'_s\subseteq N'_s=H_{PR_P}^{i-\dim R/P}(R_P)
\end{equation}
where $L'_j=(L_j^{\vee})_P^{\vee_P}$,  $N'_j=(N_j^{\vee})_P^{\vee_P}$, $N'_j/L'_j$ is $F$-nilpotent and each $L'_j/N'_{j-1}$ is either $0$ or simple as an $R_P\{F\}$-module with nontrivial $F$-action by Lemma \ref{matlis dual induce equivalence} again. And we notice that
\begin{eqnarray*}
\label{chain of equivalence relations}
&&N'_j/L'_j=0, \forall P\in \Spec R-\{\m\}\\
&\Leftrightarrow&(N_j^{\vee})_P^{\vee_P}/(L_j^{\vee})_P^{\vee_P}=0, \forall P\in \Spec R-\{\m\}\\
&\Leftrightarrow& \ker((N_j^{\vee})_P\twoheadrightarrow(L_j^{\vee})_P)=0, \forall P\in\Spec R-\{\m\}\\
&\Leftrightarrow& l_R(\ker(N_j^{\vee}\twoheadrightarrow L_j^{\vee}))<\infty\\
&\Leftrightarrow& l_R(N_j/L_j)<\infty
\end{eqnarray*}

$R_P$ is stably FH-finite for every $P\in \Spec R-\{\m\}$ if and only if $H_{PR_P}^{i-\dim R/P}(R_P)$ is anti-nilpotent for every $0\leq i\leq d$ and every $P\in \Spec R-\{\m\}$. This is because when $0\leq i\leq d$, $i-\dim R/P$ can take all values between $0$ and $\dim R_P$, and if $i-\dim R/P$ is out of this range, then the local cohomology is $0$ so it is automatically anti-nilpotent. By Proposition \ref{characterization of anti-nilpotency interms of Lyubeznik's filtration} $(2)$, this is equivalent to the condition that for every $P\in \Spec R-\{\m\}$, the corresponding $N'_j/L'_j$ is $0$. By the above chain of equivalence relations, this is equivalent to the condition that each $N_j/L_j$ have finite length as an $R$-module. By Proposition \ref{characterization of anti-nilpotency interms of Lyubeznik's filtration} $(1)$, this is equivalent to the condition that $R$ have FH-finite length.
\end{proof}

\begin{corollary}
\label{F-pure on the punctured spec implies FH-finite length for complete F-finite}
Let $(R,\m)$ be a complete and $F$-finite local ring. If $R_P$ is $F$-pure for every $P\neq\m$, then $R$ has FH-finite length.
\end{corollary}
\begin{proof}
This is clear from Theorem \ref{F-pure implies FH-finite} and Theorem \ref{FH-finite length is equiv to stably FH-finite on punctured spec}.
\end{proof}

We end this section by proving the dual result of Theorem \ref{F-pure implies FH-finite} and Corollary \ref{F-pure on the punctured spec implies FH-finite length for complete F-finite} for Cartier modules. Our result holds for a large class of $F$-finite rings (not necessarily local). So from now on till the end of this section, we only require that $R$ be an $F$-finite Noetherian ring of equal characteristic $p>0$.

We first recall an important result in \cite{BlickleBoeckleCartierModulesFiniteness} on Cartier modules. This can be viewed as the dual form of Theorem \ref{Lyubeznik's result} in the local $F$-finite case. However, we emphasize here that this result holds for all $F$-finite rings (not necessarily local).
\begin{theorem}[cf. Proposition 4.23 in \cite{BlickleBoeckleCartierModulesFiniteness}]
\label{Cartier module filtration}
Let $M$ be a Cartier module which is Noetherian over an $F$-finite ring $R$. Then there exists a finite sequence
\begin{equation}
\label{Cartier filtration}
M=M_0\supseteq\underline{M_0}\supseteq M_1\supseteq \underline{M_1}\supseteq\cdots\supseteq M_s\supseteq\underline{M_s}=0
\end{equation}
such that each $M_i/\underline{M_i}$ is nilpotent and each $\underline{M_i}/M_{i+1}$ is simple with a non-trivial $C$-action.
\end{theorem}

\begin{remark}
\label{correspondence between Cartier filtration and Lyubeznik's filtration}
When $M$ is a Cartier module which is Noetherian over an $F$-finite {\it local} ring $R$, it is straightforward to check (or using Theorem \ref{matlis dual induce equivalence}) that the Matlis dual of the sequence (\ref{Cartier filtration}) gives a filtration of $M^{\vee}$ by $R\{F\}$-modules as (\ref{Lyubeznik's filtration}) in Theorem \ref{Lyubeznik's result} (basically, $N_i=\ker(M_0^\vee\rightarrow \underline{M_i}^\vee)$ and $L_i=\ker(M_0^\vee\rightarrow M_i^\vee)$). Moreover, it is easy to see that $(M_i/\underline{M_i})^{\vee}$ corresponds to $N_i/L_i$.
\end{remark}

For any $F$-finite ring $R$, by results of Gabber \cite{Gabber.tStruc}, there exists a (normalized) dualizing complex $\omega_R^{\bullet}$ of $R$ (we refer to \cite{HartshorneResidues} for a detailed definition of the normalized dualizing complex). We define the canonical module $\omega_R$ to be $h^{-d}\omega_R^\bullet$ where $d=\dim R$. When $R$ is Cohen-Macaulay, $\omega_R$ is the only non-trivial cohomology of $\omega_R^\bullet$. And when $(R,\m)$ is $F$-finite and local, this is the usual definition of canonical module (i.e., a finitely generated module whose Matlis dual is $H_\m^d(R)$).

Applying the results in \cite{HartshorneResidues} to the Frobenius map $F$: $R\rightarrow R$, we know that there exists a canonical trace map $F_*F^{!}\omega_R^\bullet \rightarrow \omega_R^\bullet$ where $F_*$ denotes the Frobenius pushforward and $F^{!}\omega_R^\bullet=\mathbf{R}\Hom_R(F_*R, \omega_R^\bullet)$ (see also section 2.4 in \cite{BlickleBoeckleCartierModulesFiniteness}). In particular, we see that if $F^{!}\omega_R^\bullet\cong\omega_R^\bullet$, then there exists a canonical Cartier module structure on $\omega_R$, and when we localize at any prime $P\in \Spec R$, this Cartier module structure is the dual of the canonical $R_P\{F\}$-module structure of $H_{PR_P}^{\height P}(R_P)$.

\begin{remark}
To the best of our knowledge, it is still unknown whether $F^{!}\omega_R^\bullet\cong\omega_R^\bullet$ is true for all $F$-finite rings. It is always true that $F^{!}\omega_R^\bullet$ is a dualizing complex of $R$ (see \cite{HartshorneResidues}), so in our case $F^{!}\omega_R^\bullet$ differs from $\omega_R^\bullet$ by tensoring a rank 1 projective module. However, in many cases we do have $F^{!}\omega_R^\bullet\cong\omega_R^\bullet$. In \cite{BlickleBoeckleCartierModulesFiniteness}, it is proved that for an $F$-finite ring $R$, $F^{!}\omega_R^\bullet\cong\omega_R^\bullet$ holds if either $R$ is essentially of finite type over a Gorenstein local ring or $R$ is sufficiently affine. We refer to Proposition 2.20 and Proposition 2.21 in \cite{BlickleBoeckleCartierModulesFiniteness} as well as \cite{HartshorneResidues} for more details on this question.
\end{remark}

\begin{theorem}
\label{dual form of main result}
Let $R$ be an $F$-finite ring of dimension $d$. Let $\omega_R^\bullet$ be the normalized dualizing complex of $R$ and $\omega_R\cong h^{-d}\omega_R^\bullet$ be the canonical module of $R$. Suppose $F^{!}\omega_R^\bullet\cong\omega_R^\bullet$ (for example if $R$ is essentially of finite type over an $F$-finite field), then we have the following:
\begin{enumerate}
\item If $R$ is $F$-pure, then $\omega_R$ only has finitely many Cartier submodules.
\item If $R_P$ is $F$-pure for every non-maximal prime $P$, then $\omega_R$ has finite length in the category of Cartier modules.
\end{enumerate}
\end{theorem}
\begin{proof}
By the discussion above, there exists a canonical Cartier module structure on $\omega_R$. So by Theorem \ref{Cartier module filtration}, we have a finite sequence of Cartier modules:
\begin{equation}
\label{cartier module filtration}
\omega_R=M_0\supseteq\underline{M_0}\supseteq M_1\supseteq \underline{M_1}\supseteq\cdots\supseteq M_s\supseteq\underline{M_s}=0.
\end{equation}
For every maximal ideal $\m\in\Spec R$, the above sequence localizes to give a corresponding sequence of Cartier modules over $R_\m$:
\begin{equation}
\label{localized cartier module filtration}
\omega_{R_\m}=(M_0)_\m\supseteq(\underline{M_0})_\m\supseteq (M_1)_\m\supseteq (\underline{M_1})_\m\supseteq\cdots\supseteq (M_s)_\m\supseteq(\underline{M_s})_\m=0.
\end{equation}

If $R$ is $F$-pure, then we claim that in (\ref{cartier module filtration}), $M_i/\underline{M_i}=0$ for every $i$. If not, pick $\m\in\Spec R$ such that $(M_i/\underline{M_i})_\m\neq0$. We look at the corresponding sequence of Cartier modules in (\ref{localized cartier module filtration}). Since $R_\m$ is local, by Remark \ref{correspondence between Cartier filtration and Lyubeznik's filtration}, the Matlis dual of this sequence gives a filtration of $\omega_{R_\m}^{\vee_m}$ of $R_\m\{F\}$-modules which corresponds to the filtration (\ref{Lyubeznik's filtration}) in Theorem \ref{Lyubeznik's result}. In particular, we know that $\omega_{R_\m}^{\vee_m}\cong H_\m^d(R_\m)$ has a filtration as in Theorem \ref{Lyubeznik's result} with some $N_i/L_i\neq0$. This shows that $H_\m^d(R_\m)$ is not anti-nilpotent by Proposition \ref{characterization of anti-nilpotency interms of Lyubeznik's filtration}. But this contradicts the condition that $R_\m$ be $F$-split by Theorem \ref{F-split implies anti-nilpotency} (note that here $F$-pure is the same as $F$-split since $R_\m$ is $F$-finite). A very similar argument shows that if $R_P$ is $F$-pure for every non-maximal prime $P\in\Spec R$, then $M_i/\underline{M_i}$ is supported only at the maximal ideals of $R$ for every $i$, in particular, each $M_i/\underline{M_i}$ has finite length as an $R$-module. Combine with (\ref{cartier module filtration}), we already proved (2).

For (1), notice that $M_i/\underline{M_i}=0$ for every $i$ implies that $\omega_R$ has a Cartier module filtration such that each successive quotient is a simple Cartier module with non-trivial (in particular, surjective) $C$-action. It follows that every Cartier submodule $N$ of $\omega_R$ satisfies $C(N)=N$. Now it follows from Corollary 4.20 in \cite{BlickleBoeckleCartierModulesFiniteness} that $\omega_R$ only has finitely many Cartier submodules.

\end{proof}

\section{F-purity on the punctured spectrum implies FH-finite length for excellent local rings}
In this section we prove that for excellent local rings, $F$-purity on the punctured spectrum implies FH-finite length. In fact, if we assume $R$ is complete and $F$-finite, this is exactly Corollary \ref{F-pure on the punctured spec implies FH-finite length for complete F-finite}. Our main point is to make use of the $\Gamma$-construction introduced in \cite{HochsterHunekeFRegularityTestElementsBaseChange} to prove the general case. We also prove that the properties such as having FH-finite length, being FH-finite, and being stably FH-finite localize. We start with a brief review on the $\Gamma$-construction. We refer to \cite{HochsterHunekeFRegularityTestElementsBaseChange} for details.

Let $K$ be a field of positive characteristic $p>0$ with a $p$-base $\Lambda$. Let $\Gamma$ be a fixed cofinite subset of $\Lambda$. For $e\in \mathbb{N}$ we denote by $K^{\Gamma,e}$ the purely inseparable field extension of $K$ that is the result of adjoining $p^e$-th roots of all elements in $\Gamma$ to $K$.

Now suppose that $(R,\m)$ is a complete local ring with $K\subseteq R$ a coefficient field. Let $x_1,\dots,x_d$ be a system of parameters for $R$, so that $R$ is module-finite over $A=K[[x_1,\dots,x_d]]\subseteq R$. Let $A^\Gamma$ denote \[\bigcup_{e\in\mathbb{N}}K^{\Gamma,e}[[x_1,\dots,x_d]],\] which is a regular local ring that is faithfully flat and purely inseparable over $A$. The maximal ideal of $A$ expands to that of $A^{\Gamma}$. Let $R^{\Gamma}$ denote $A^{\Gamma}\otimes_AR$, which is module-finite over the regular ring $A^{\Gamma}$ and is faithfully flat and purely inseparable over $R$. The maximal ideal of $R$ expands to the maximal ideal of $R^{\Gamma}$. The residue field of $R^{\Gamma}$ is $K^{\Gamma}=\bigcup_{e\in\mathbb{N}}K^{\Gamma,e}$. It is of great importance that $R^{\Gamma}$ is $F$-finite. Moreover, we can preserve some good properties of $R$ if we choose a sufficiently small cofinite subset $\Gamma$:
\begin{lemma}[{\it cf.} Lemma 6.13 in \cite{HochsterHunekeFRegularityTestElementsBaseChange}]
\label{gamma construction preserves prime}
Let $R$ be a complete local ring. If $P$ is a prime ideal of $R$ then there exists a cofinite set $\Gamma_0\subseteq \Lambda$ such that $Q=PR^\Gamma$ is a prime ideal in $R^\Gamma$ for all $\Gamma\subseteq \Gamma_0$.
\end{lemma}
\begin{lemma}[{\it cf.} Lemma 2.9 and Lemma 4.3 in \cite{MelandFlorianFrobenius}]
\label{gamma construction preserves F-injective}
Let $R$ be a complete local ring. Let $W$ be an $R\{F\}$-module that is Artinian as an $R$-module such that the $F$-action is injective. Then for any sufficiently small choice of $\Gamma$ cofinite in $\Lambda$, the action of $F$ on $R^{\Gamma}\otimes_RW$ is also injective. Moreover, if $R^{\Gamma}$ is FH-finite (resp. has FH-finite length), then so is $R$.
\end{lemma}

Now we start proving our main theorems. We first show that $F$-purity is preserved under nice base change. This is certainly well-known to experts. We refer to \cite{HashimotoF-purehomomorphisms} and \cite{SchwedeandZhangBertinitheoremsforFsingularities} for some (even harder) results on base change problems. Since in most of these references the results are only stated for $F$-finite rings, we provide a proof which works for all excellent rings. In fact, the proof follows from essentially the same argument as in the proof of Theorem 7.24 in \cite{HochsterHunekeFRegularityTestElementsBaseChange}. First we recall the notion of {\it Frobenius closure}: for any ideal $I\subseteq R$, $I^F=\{x\in R|x^{p^e}\in I^{[p^e]}$ for some $e\}$. If $R$ is $F$-pure, then every ideal is Frobenius closed. We will see that under mild conditions on the ring, the converse also holds.

We also need the notion of {\it approximately Gorenstein ring } introduced in \cite{HochsterCyclicPurity}: $(R,\m)$ is approximately Gorenstein if there exists a decreasing sequence of $\m$-primary ideals $\{I_t\}$ such that every $R/I_t$ is a Gorenstein ring  and the $\{I_t\}$ are cofinal with the powers of $\m$. That is, for every $N>0$, $I_t\subseteq \m^N$ for all $t\gg 1$. We will call such a sequence of ideals an {\it approximating sequence of ideals}. Note that for an $\m$-primary ideal $I$, $R/I$ is Gorenstein if and only if $I$ is an irreducible ideal, i.e. it is not the intersection of two strictly larger ideals. Every reduced excellent local ring is approximately Gorenstein (see \cite{HochsterCyclicPurity}).
\begin{lemma}
\label{characterization of F-pure in terms of Frobenius closure}
Let $(R,\m)$ be an excellent local ring. Then $R$ is $F$-pure if and only if there exists an approximating sequence of ideals $\{I_t\}$ such that $I_t^F=I_t$.
\end{lemma}
\begin{proof}
If $(R,\m)$ is excellent and $F$-pure then $R$ is reduced, hence approximately Gorenstein. So there exists an approximating sequence of ideals $\{I_t\}$. $I_t^F=I_t$ follows because $R$ is $F$-pure. 
For the other direction, we use $R^{(1)}$ to denote the target ring under the Frobenius endomorphism $R\rightarrow R^{(1)}$ and we want to show $R\rightarrow R^{(1)}$ is pure. It suffices to show that $E_R\hookrightarrow R^{(1)}\otimes_RE_R$ is injective where $E_R$ denotes the injective hull of the residue field of $R$. But it is easy to check that $E_R=\D\varinjlim_t\frac{R}{I_t}$. Hence $E_R\hookrightarrow R^{(1)}\otimes_RE_R$ is injective if $\D\frac{R}{I_t}\hookrightarrow \frac{R^{(1)}}{I_tR^{(1)}}$ is injective for all $t$. But this is true because $I_t^F=I_t$.
\end{proof}
\begin{proposition}
\label{F-pure preserved under flat map with regular fibres}
Let $(R,\m)\rightarrow (S,\n)$ be a faithfully flat extension of excellent local rings such that the closed fibre $S/\m S$ is Gorenstein and $F$-pure. If $R$ is $F$-pure, then $S$ is $F$-pure.
\end{proposition}
\begin{proof}
Let $\{I_k\}$ be an approximating sequence of ideals in $R$. Let $x_1,\dots,x_n$ be elements in $S$ such that their image form a system of parameters in $S/\m S$. Since $S/\m S$ is Gorenstein, $(x_1^t,\dots,x_n^t)$ is an approximating sequence of ideals in $S/\m S$. So that $I_k+(x_1^t,\dots,x_n^t)$ is an approximating sequence of ideals in $S$ (see the proof of Theorem 7.24 in \cite{HochsterHunekeFRegularityTestElementsBaseChange}).

To show that $S$ is $F$-pure, it suffices to show every $I_k+(x_1^t,\dots,x_n^t)$ is Frobenius closed by Lemma \ref{characterization of F-pure in terms of Frobenius closure}. Therefore we reduce to showing that if $I$ is an irreducible ideal primary to $\m$ in $R$ and $x_1,\dots,x_n$ are elements in $S$ such that their image form a system of parameters in $S/\m S$, then $(IS+(x_1,\dots,x_n))^F=IS+(x_1,\dots,x_n)$ in $S$.

Let $v$ and $w$ be socle representatives of $R/I$ and $(S/\m S)/(x_1,\dots,x_n)(S/\m S)$ respectively. It suffices to show $vw\notin (IS+(x_1,\dots,x_n))^F$. Suppose we have \[v^qw^q\in I^{[q]}S+(x_1^q,\dots,x_n^q).\] Taking their image in $S/(x_1^q,\dots,x_n^q)$, we have \[\overline{v}^q\overline{w}^q\in \overline{I^{[q]}S},\] therefore \[\overline{w}^q\in (\overline{I^{[q]}S}:\overline{v}^q)=(I^{[q]}:v^q)(S/(x_1^q,\dots,x_n^q)S)\] where the second equality is because $S/(x_1^q,\dots,x_n^q)S$ is faithfully flat over $R$. So we have \[w^q\in (I^{[q]}:v^q)S+(x_1^q,\dots,x_n^q).\] Since $R$ is $F$-pure, $(I^{[q]}:v^q)\in \m$. So after taking image in $S/\m S$, we get that \[w^q\in(x_1^q,\dots,x_n^q)(S/\m S).\] Hence $w\in (x_1^q,\dots,x_n^q)^F$ in $S/\m S$, this contradict to the condition that $S/\m S$ be $F$-pure.
\end{proof}

Next we want to show the openness of {\it $F$-pure locus} when $R$ is $F$-finite, where the $F$-pure locus is defined as the set $\{P\in \Spec R| R_P $ is $F$-pure $\}$.
\begin{lemma}
\label{openness of F-pure locus}
Let $(R,\m)$ be an $F$-finite local ring. Then the $F$-pure locus is open in $\Spec R$ in the Zariski topology.
\end{lemma}
\begin{proof}
We will use $R^{(1)}$ to denote the target ring under the Frobenius endomorphism $R\rightarrow R^{(1)}$. Since $R$ is $F$-finite, $R^{(1)}$ is finitely generated as an $R$-module.

It suffice to show that if $R_P$ is $F$-pure, then there exists $x\notin P$ such that $R_x$ is $F$-pure. Since $R$ is $F$-finite, so is $R_P$. Therefore $R_P$ is $F$-split. That is, the map $R_P\hookrightarrow R^{(1)}_P$ splits as $R$-modules. This implies that the map \[\Hom_{R_P}(R^{(1)}_P, R_P)\rightarrow \Hom_{R_P}(R_P, R_P)\] is surjective. Since $R^{(1)}$ is finitely generated as an $R$-module, $\Hom_{R_P}(R^{(1)}_P, R_P)\cong\Hom_R(R^{(1)}, R)_P$. So we know that \[\Hom_{R}(R^{(1)}, R)_P\rightarrow \Hom_{R}(R, R)_P\] is surjective. Hence we may pick $x\notin P$ such that \[\Hom_{R}(R^{(1)}, R)_x\rightarrow \Hom_{R}(R, R)_x\] is surjective. This proves that $R_x$ is $F$-split, hence $F$-pure.
\end{proof}

Now we show that $F$-purity on the punctured spectrum is preserved under the $\Gamma$-construction when we pick $\Gamma$ sufficiently small and cofinite in $\Lambda$.
\begin{proposition}
\label{use gamma construction to get F-pure on punctured spec}
Let $(R,\m)$ be a complete local ring such that $R_P$ is $F$-pure on the punctured spectrum $\Spec{R}-\{\m\}$. Then for any sufficiently small choice of $\Gamma$ cofinite in $\Lambda$, $R^\Gamma$ is $F$-pure on the punctured spectrum $\Spec{R^\Gamma}-\{\m R^\Gamma\}$.
\end{proposition}
\begin{proof}
Because $R^\Gamma$ is purely inseparable over $R$, for all $P\in \Spec R$ there is a unique prime ideal $P^{\Gamma}\in \Spec {R^{\Gamma}}$ lying over $P$. In particular, $\Spec{R^\Gamma}\cong\Spec R$. Since $R^{\Gamma}$ is $F$-finite, we know the $F$-pure locus of each $R^{\Gamma}$, call it $X_\Gamma$, is open in $\Spec{R^\Gamma}=\Spec R$ by Lemma \ref{openness of F-pure locus}. Since open sets in $\Spec R$ satisfy ACC, we know that there exists $\Gamma$ such that $X_\Gamma$ is maximal. We will show that $X_\Gamma\supseteq \Spec R-\{\m\}$. This will prove $R^\Gamma$ is $F$-pure on the punctured spectrum $\Spec{R^\Gamma}-\{\m R^\Gamma\}$.

Suppose there exists $Q\neq \m$ such that $Q\notin X_\Gamma$. We may pick $\Gamma'\subseteq \Gamma$ sufficiently small and cofinite in $\Lambda$ such that $QR^{\Gamma'}$ is prime (that is, $QR^{\Gamma'}=Q^{\Gamma'}$) by Lemma \ref{gamma construction preserves prime}. So $R_Q\rightarrow R^{\Gamma'}_{Q^{\Gamma'}}$ is faithfully flat whose closed fibre is a field. By Proposition \ref{F-pure preserved under flat map with regular fibres}, $R^{\Gamma'}_{Q^{\Gamma'}}$ is $F$-pure. Since $\Gamma'\subseteq \Gamma$, $R^{\Gamma'}\rightarrow R^{\Gamma}$ is faithfully flat so $R^{\Gamma'}_{P^{\Gamma'}}\rightarrow R^\Gamma_{P^\Gamma}$ is faithfully flat for each $P\in \Spec R$. Now for $P\in X_\Gamma$, $R^\Gamma_{P^\Gamma}$ is $F$-pure, hence so is $R^{\Gamma'}_{P^{\Gamma'}}$. So $X_{\Gamma'}\supseteq X_{\Gamma}\cup \{Q\}$, which is a contradiction since we assume that $X_\Gamma$ is maximal.
\end{proof}

\begin{theorem}
Let $(R,\m)$ be an excellent local ring such that $R_P$ is $F$-pure for every $P\in \Spec R - \{\m\}$. Then $R$ has FH-finite length.
\end{theorem}
\begin{proof}
We look at the chain of faithfully flat ring extensions: \[R\rightarrow \widehat{R}\rightarrow \widehat{R}^\Gamma\rightarrow \widehat{\widehat{R}^\Gamma}.\] Since $R$ is excellent, for every $Q_0\in\Spec{\widehat{R}}-\{\m\widehat{R}\}$ lying over $P$ in $R$, $R_P\rightarrow \widehat{R}_{Q_0}$ has geometrically regular fibres, so $\widehat{R}_{Q_0}$ is $F$-pure by Proposition \ref{F-pure preserved under flat map with regular fibres}. Hence $\widehat{R}$ is $F$-pure on the punctured spectrum. So by Proposition \ref{use gamma construction to get F-pure on punctured spec}, we can pick $\Gamma$ sufficiently small and cofinite in $\Lambda$ such that $\widehat{R}^\Gamma$ is $F$-pure on the punctured spectrum.

For every $Q\in\Spec{\widehat{\widehat{R}^\Gamma}}-\{\m\widehat{\widehat{R}^\Gamma}\}$, let $Q_1\neq \m{\widehat{R}^\Gamma}$  be the contraction of $Q$ to $\widehat{R}^\Gamma$. Since $\widehat{R}^\Gamma_{Q_1}$ is $F$-finite, it is excellent by \cite{KunzOnNoetherianRingsOfCharP}. So the closed fibre of $\widehat{R}^\Gamma_{Q_1}\rightarrow \widehat{\widehat{R}^\Gamma}_Q$ is geometrically regular. So $\widehat{\widehat{R}^\Gamma}_Q$ is $F$-pure by Proposition \ref{F-pure preserved under flat map with regular fibres}.

It follows that, for sufficiently small choice of $\Gamma$ cofinite in $\Lambda$, $\widehat{\widehat{R}^\Gamma}_Q$ is $F$-pure for every $Q\in\Spec{\widehat{\widehat{R}^\Gamma}}-\{\m\widehat{\widehat{R}^\Gamma}\}$. Now by Theorem \ref{F-pure implies FH-finite} applied to $\widehat{\widehat{R}^\Gamma}_Q$ and Theorem \ref{FH-finite length is equiv to stably FH-finite on punctured spec} applied to $\widehat{\widehat{R}^\Gamma}$, we know $\widehat{\widehat{R}^\Gamma}$ has FH-finite length. Hence so does $R$ by Lemma \ref{reduce to F-split} and Lemma \ref{gamma construction preserves F-injective}.
\end{proof}

\begin{lemma}
\label{gamma construction preserves simple}
Let $R$ be a complete local ring. Let $W$ be a simple $R\{F\}$-module that is Artinian over $R$ with nontrivial $F$-action. Then for any sufficiently small choice of $\Gamma$ cofinite in $\Lambda$, $W\otimes_RR^\Gamma$ is a simple $R^\Gamma\{F\}$-module with nontrivial $F$-action.
\end{lemma}
\begin{proof}
By Lemma \ref{gamma construction preserves F-injective}, we may pick $\Gamma$ sufficiently small and cofinite in $\Lambda$ such that $F$ acts injectively on $W\otimes_RR^{\Gamma}$. I claim such a $W\otimes_RR^{\Gamma}$ must be a simple $R^\Gamma\{F\}$-module. If not, then by Theorem \ref{Lyubeznik's result}, we have $0\subsetneqq L \subsetneqq W\otimes_RR^{\Gamma}$ where $L$ is simple with nontrivial Frobenius action. Now we pick $0\neq x\in L$. Because $R^{\Gamma}$ is purely inseparable over $R$, there exists $e$ such that $0\neq F^{e}(x)\in W$. Hence $L\cap W\neq 0$. But it is straightforward to check that $L\cap W$ is an $R\{F\}$-submodule of $W$. So $L\cap W=W$ since $W$ is simple. Hence $L\supseteq W\otimes_RR^\Gamma$ which is a contradiction.
\end{proof}

\begin{proposition}
\label{FH-finite length preserved under gamma construction}
Let $(R,\m)$ be a complete local ring. Then
\begin{enumerate}
\item If $R$ has FH-finite length , then so does $R^\Gamma$ for $\Gamma$ sufficiently small and cofinite in $\Lambda$.
\item If $R$ is stably FH-finite, then so is $R^\Gamma$ for $\Gamma$ sufficiently small and cofinite in $\Lambda$.
\end{enumerate}
\end{proposition}
\begin{proof}
By Theorem \ref{Lyubeznik's result}, for every $0\leq i\leq d=\dim R$, we have a filtration of $R\{F\}$-modules \[0= L_0\subseteq N_0\subseteq L_1\subseteq N_1\subseteq \cdots\subseteq L_s\subseteq N_s=H_\m^i(R)\] with each $N_j/L_j$ $F$-nilpotent and $L_j/L_{j-1}$ simple as an $R\{F\}$-module with nonzero $F$-action. By Lemma \ref{gamma construction preserves simple}, we can pick $\Gamma$ sufficiently small and cofinite in $\Lambda$ such that, for all $i$, all $L_j/L_{j-1}\otimes_RR^\Gamma$ are simple with nonzero $F$-action. Hence \[0= L_0^\Gamma\subseteq N_0^\Gamma\subseteq L_1^\Gamma\subseteq N_1^\Gamma\subseteq \cdots\subseteq L_s^\Gamma\subseteq N_s^\Gamma=H_\m^i(R^\Gamma)\] where $L_j^\Gamma=L_j\otimes_RR^\Gamma$ and $N_j^\Gamma=N_j\otimes_RR^\Gamma$ is a corresponding filtration of $H_\m^i(R^\Gamma)$. Now both (1) and (2) are clear from Proposition \ref{characterization of anti-nilpotency interms of Lyubeznik's filtration}.
\end{proof}

\begin{theorem}
\label{FH-finite localize}
Let $(R,\m)$ be a local ring that has FH-finite length (resp. is FH-finite or stably FH-finite). Then the same holds for $R_P$ for every $P\in \Spec{R}$.
\end{theorem}
\begin{proof}
It suffices to show that if $(R,\m)$ has FH-finite length, then $R_P$ is stably FH-finite for every $P\neq \m$. We first notice that $\widehat{R}$ has FH-finite length by Lemma \ref{reduce to F-split}. We pick $\Gamma$ sufficiently small and cofinite in $\Lambda$ such that $\widehat{R}^\Gamma$ still has FH-finite length by Proposition \ref{FH-finite length preserved under gamma construction}. Now we complete again, and we get that $B=\widehat{\widehat{R}^\Gamma}$ is an $F$-finite complete local ring that has FH-finite length by Lemma \ref{reduce to F-split}, and the maximal ideal in $B$ is $\m B$. Notice that $R\rightarrow B$ is faithfully flat, hence for every $P\neq \m$, we may pick $Q\in \Spec {B}-\{\m B\}$ such that $R_P\rightarrow B_Q$ is faithfully flat (in particular, pure) and $P B_Q$ is primary to $Q B_Q$. So $\widehat{R_P}\rightarrow \widehat{B_Q}$ is split. By Theorem \ref{FH-finite length is equiv to stably FH-finite on punctured spec} applied to $B$, $B_Q$ is stably FH-finite. Hence so is $\widehat{B_Q}$ by Lemma \ref{reduce to F-split}. Now we apply Corollary \ref{FH-finite for split couple}, we see that $\widehat{R_P}$ is stably FH-finite. Hence so is $R_P$ by Lemma \ref{reduce to F-split}.
\end{proof}

\section{some examples}
Since stably FH-finite trivially implies $F$-injective, it is quite natural to ask whether FH-finite implies $F$-injective. The following example studied in \cite{MelandFlorianFrobenius} shows this does not hold in general.
\begin{example}[{\it cf.} Example 2.15 in \cite{MelandFlorianFrobenius}]
Let $R=K[[x,y,z]]/(x^3+y^3+z^3)$ where $K$ is a field of characteristic different from $3$. This is a Gorenstein ring of dimension $2$. And it can be checked that the only nontrivial $F$-compatible submodule in $H_\m^2(R)$ is its socle, a copy of $K$. Hence $R$ is FH-finite. But it is known that $R$ is $F$-pure (equivalently, $F$-injective since $R$ is Gorenstein) if and only if the characteristic of $K$ is congruent to $1$ mod $3$. Hence if the characteristic is congruent to $2$ mod $3$, we get an example of FH-finite ring which is not $F$-injective.
\end{example}

Another natural question to ask is whether the converse of Theorem
\ref{F-pure implies FH-finite} is true. The next example will show
this is also false in general. We recall a theorem in
\cite{SmithFRatImpliesRat}:
\begin{theorem}[{\it cf.} Theorem 2.6 in \cite{SmithFRatImpliesRat} or Proposition 2.12 in \cite{MelandFlorianFrobenius}]
\label{smith's theorem} Let $(R,\m)$ be an excellent Cohen-Macaulay
local ring of dimension $d$. Then $R$ is $F$-rational if and only if
$H_\m^d(R)$ is a simple $R\{F\}$-module.
\end{theorem}
\begin{corollary}
\label{F-rational implies stably FH-finite} Let $(R,\m)$ be an
excellent $F$-rational local ring of dimension $d$. Then $R$ is
stably FH-finite.
\end{corollary}
\begin{proof}
This follows immediately from Theorem \ref{smith's theorem} and
Theorem \ref{anti-nilpotency is equiv to stably FH finite} because
when $H_\m^d(R)$ is a simple $R\{F\}$-module, it is obviously
anti-nilpotent.
\end{proof}
\begin{example}[{\it cf.} Example 7.15 in \cite{HochsterHunekeTightClosureOfParameterIdealsAndSplitting}]
\label{F-pure is stronger than stably FH-finite} Let $R=K[t, xt^4,
x^{-1}t^4, (x+1)^{-1}t^4]_{\m}\subseteq K(x,t)$ where $\m=(t, xt^4,
x^{-1}t^4, (x+1)^{-1}t^4)$. Then $R$ is $F$-rational but not
$F$-pure. Hence by our Corollary \ref{F-rational implies stably
FH-finite}, $R$ is a stably FH-finite Cohen-Macaulay ring that is
not $F$-pure.
\end{example}

Also note that even when $R$ is Cohen-Macaulay and $F$-injective, it
is not always FH-finite, and does not even always have FH-finite length. We have the following example:
\begin{example}[{\it cf.} Example 2.16 in \cite{MelandFlorianFrobenius}]
\label{stably FH-finite is stronger than F-injective} Let $k$ be an
infinite perfect field of characteristic $p>2$, $K=k(u,v)$, where
$u$ and $v$ are indeterminates, and let $L=K[y]/(y^{2p}+uy^p-v)$.
Let $R=K+xL[[x]]\subseteq L[[x]]$. Then $R$ is a complete
$F$-injective Cohen-Macaulay domain of dimension $1$ which is not
FH-finite. Notice that by Theorem \ref{anti-nilpotency is equiv to stably FH finite}, $R[[x]]$ is an $F$-injective Cohen-Macaulay domain of dimension $2$ that does not have FH-finite length (this was not pointed out in \cite{MelandFlorianFrobenius}).
\end{example}

\section*{Acknowledgement} I would like to thank Mel Hochster for many helpful and valuable discussions on the problem. I would also like to thank Alberto F. Boix, Florian Enescu, Karl Schwede, Rodney Sharp and Wenliang Zhang for their interest in this work and for their valuable comments. Finally I would like to thank the referee for suggesting the dual form of the main theorems which leads to Theorem \ref{dual form of main result}.

\bibliographystyle{skalpha}
\bibliography{CommonBib}

\end{document}